\documentclass[12pt,letterpaper]{article}
\usepackage[utf8]{inputenc}

\usepackage{amsmath,graphicx}
\usepackage{amsthm,amssymb}
\usepackage{thm-restate}
\usepackage{tikz}
\usepackage{tikz-cd}
\usepackage{mathtools}
\usepackage[colorlinks=true,allcolors=blue]{hyperref}
\usepackage{cleveref}
\usepackage{nicefrac}
\usepackage{caption}
\usepackage[labelfont={}]{subcaption}
\usepackage{xspace}
\usepackage{verbatim}
\usepackage[left=2.4cm,right=2.4cm,top=2.4cm,bottom=2.4cm]{geometry}

\newtheorem{theorem}{Theorem}

\newtheorem{lemma}[theorem]{Lemma}

\theoremstyle{remark}
\newtheorem{case}{Case}
\numberwithin{case}{theorem}

\numberwithin{subcase}{case}

\crefname{figure}{Figure}{Figures}

\DeclarePairedDelimiter\ceil{\lceil}{\rceil}
\DeclarePairedDelimiter\floor{\lfloor}{\rfloor}

\DeclareMathOperator{\crg}{cr}

\newcommand*\circled[1]{\tikz[baseline=(char.base)]{
		\node[shape=circle,draw,inner sep=.5pt](char){\fontsize{4pt}{0cm}\selectfont #1};}}
\newcommand{\crN}[1]{\crg_{\circled{$#1$}}}

\newcommand{\vtx}[4]{#1_{#2}(\textsc{#3,#4})}

\title{The tripartite-circle crossing number\\ of graphs with two small partition classes}
\author{
	Charles Camacho\thanks{{University of Washington, USA, \tt camachoc@uw.edu}},
	Silvia Fern{\'a}ndez-Merchant\thanks{{California State University, Northridge, USA, \tt silvia.fernandez@csun.edu}},
	Marija Jeli{\'c} Milutinovi{\'c}\thanks{{University of Belgrade, Serbia,  \tt marija.jelic@matf.bg.ac.rs}}\\
	Rachel Kirsch\thanks{{George Mason University,	USA, \tt rkirsch4@gmu.edu}},
	Linda Kleist\thanks{{TU Braunschweig, Germany, \tt l.kleist@tu-bs.de}},
	Elizabeth Bailey Matson\thanks{{Alfred University, USA,  \tt matson@alfred.edu}},
	Jennifer White\thanks{{Saint Vincent College, USA, \tt jennifer.white@stvincent.edu}}
}
\date{}

\begin{document}
	\maketitle
	\begin{abstract}
		A tripartite-circle drawing of a tripartite graph is a drawing in the plane, where each part of a vertex partition is placed on one of three disjoint circles, and the edges do not cross the circles. The tripartite-circle crossing number of a tripartite graph is the minimum number of edge crossings among all its tripartite-circle drawings. We determine the exact value of the tripartite-circle crossing number of $K_{a,b,n}$, where $a,b\leq 2$.
	\end{abstract}
	
	\graphicspath{{../figures/}}
	
	\section{Introduction}
	The \emph{crossing number} of a graph $G$, denoted by $\crg(G)$, is the minimum number of crossings over all drawings of $G$ in the plane. There are multiple variants to the problem of finding the crossing number of a graph. Some problems consider drawings on other surfaces or in other spaces, like drawings on the torus, a cylinder, or a $k$-page book. Others restrict the minimum to specific types of drawings, like rectilinear or geometric drawings, where the edges must be straight line segments. Drawings with few crossings have been studied in connection with readability and VLSI chip design \cite{leighton}. See \cite{Sch14} for a survey of crossing number variants and some of their applications.
	
	Exact crossing numbers are unknown even for very special graph classes. Famous examples of these are the long-standing Harary-Hill Conjecture 
	on the crossing number of the complete graph~$K_n$ \cite{HH,G1960} and the Zarankiewicz Conjecture on the complete bipartite graph~$K_{m,n}$~\cite{Z}. Among the most studied drawings of complete graphs are $2$-page book drawings (or \emph{$1$-circle drawings}) and cylindrical drawings (or \emph{$2$-circle drawings}). This prompted the study of \emph{$k$-circle drawings} of graphs \cite{FGHLM}, that is, drawings in the plane where the vertices are placed on $k$ specified circles and the edges cannot cross these circles. Of particular interest are \emph{$k$-partite-circle drawings}, where we further require that the vertices on each circle form an independent set. The minimum number of crossings among all $k$-partite-circle drawings of a graph $G$ is known as the \emph{$k$-partite-circle crossing number} of $G$ and is denoted by $\crN{k}(G)$. 
	
	This crossing number has been studied for complete bipartite and tripartite graphs. 
	In 1997, Richter and Thomassen~\cite{RT} settled the balanced case for complete bipartite graphs by showing that
	$ \crN{2}(K_{n,n})=n\binom{n}{3}$.
	\'Abrego, Fern\'andez-Merchant, and Sparks~\cite{AFS} generalized this result to all complete bipartite graphs. The exact expression for this crossing number was complicated as it was given in terms of summations involving floor and ceiling functions. It was recently simplified  to
	\begin{equation}\label{eq:ccr_Kmn}
		\crN{2} (K_{m,n})=\binom{n}{2}\binom{m}{2}-\nicefrac{1}{12}\cdot(n^2m^2-n^2-m^2+\gcd(m,n)^2)
	\end{equation}
	by \'Abrego and Fern\'andez-Merchant in~\cite{AF}. 
	In previous work \cite{tripartite_JGT}, we proved lower and upper bounds on $\crN{3}(K_{m,n,p})$; the implied lower and upper bounds for $K_{n,n,n}$ are of orders $\nicefrac{5}{4}\cdot n^4$ and $\nicefrac{6}{4}\cdot n^4$, respectively.
	
	In this paper, we study $\crN{3}(K_{m,n,p})$ for small values of $m$ and $n$. First, for $m=n=1$, we show that crossing-minimal tripartite-circle drawings are in one-to-one correspondence with crossing-minimal bipartite-circle drawings of $K_{2,n}$.
	
	\begin{restatable}{proposition}{KOneOneN}\label{obs:K11n}
		For every positive integer $n$,  \[\crN{3} (K_{1,1,n})=\crN{2} (K_{2,n})=\big\lceil\nicefrac{1}{4}\cdot n(n-2)\big\rceil.\]
	\end{restatable}
	
	The next smallest case, $K_{1,2,n}$, behaves differently. 
	We show that $\crN{3} (K_{1,2,n})\approx \nicefrac{3}{4}\cdot n^2$, which is smaller than $\crN{2} (K_{3,n})\approx \nicefrac{5}{6}\cdot n^2$.
	
	\begin{restatable}{theorem}{KOneTwoN}\label{thm:K12n}
		For every integer $n\geq 2$,
		$$\crN{3} (K_{1,2,n})=	\big\lfloor\nicefrac{3}{4}\cdot n^2\big\rfloor-n.$$
	\end{restatable}
	
	In previous work, we established that $\crN{3}(K_{2,2,2}) = 3$; for an illustration we refer to Figure 12(c) in \cite{tripartite_JGT}. As our main result, we establish $\crN{3}(K_{2,2,n})$ for all $n \ge 3$. 
	
	\begin{restatable}{theorem}{KTwoTwoN}\label{th:K22n}
		For every integer $n\geq 3$,
		\begin{align*}
			\crN{3} (K_{2,2,n})&=
			6\bigg\lfloor\frac{n}{2}\bigg\rfloor\left\lfloor\frac{n-1}{2}\right\rfloor+2n-3.
		\end{align*}
		
	\end{restatable}
	
	In comparison, the Zarankiewicz Conjecture on $\crg(K_{m,n})$ has been proved only when $\min(m, n)\leq 6$ by Kleitman in 1970 \cite{K}, and for $m = 7$ or $m=8$ and $n\leq 10$ by Woodall in 1993 ~\cite{woodall1993}. The current best lower bounds are by de Klerk et al.~\cite{KMPRS} and Norin and Zwols~\cite{N}.
	The crossing number of the complete tripartite graph is also unknown in general. In 2017, Gethner et al.~\cite{GHLPRY} provided a drawing of $K_{m,n,p}$ with few crossings and conjectured that their drawing is crossing-optimal. The exact crossing numbers for small values of $m$ and $n$ are known only for $\crg(K_{1,n,p})$ with $n\leq 5$~\cite{Ho08,A1986,HZ,HM} and for $\crg(K_{2,n,p})$ with $n\leq 4$~\cite{KS,A1986,Ho13}. For $\crg(K_{3,3,p})$, only asymptotically tight bounds are known~\cite{GM}. 
	A general conjecture on the crossing number of $K_{1,m,n}$ has been proven only under the hypothesis that the Zarankiewicz Conjecture holds~\cite{YW}. See \cite{CHN} for a survey.
	
	A contrast between crossing numbers and circle crossing numbers can be seen in a comparison of the graphs $K_{2,2,n}$ and $K_{4,n}$, which differ by only four edges. While their crossing numbers are equal, as $\crg(K_{4,n})=\crg(K_{2,2,n})=2\lfloor \nicefrac{n}{2} \rfloor\lfloor\nicefrac{(n-1)}{2}\rfloor=\nicefrac{1}{2}\cdot n^2 +\Theta(n)$~\cite{K,KS}, their circle crossing numbers evaluate to $\crN{3}(K_{2,2,n})= \nicefrac{3}{2}\cdot n^2+\Theta(n)$ and $\crN{2}(K_{4,n})= \nicefrac{7}{4}\cdot n^2 +\Theta(n)$~\cite{AFS} and thus differ by a term quadratic in $n$.
	
	\subparagraph{Organization.}
	The remainder of our paper is organized as follows:  In \Cref{sec:2}, we introduce tools for counting the number of crossings based on previous work. 
	In \Cref{sec:K12n}, we prove \Cref{obs:K11n} regarding crossing-minimal tripartite-circle drawings of $K_{1,1,n}$ and prove \Cref{thm:K12n} determining the tripartite-circle crossing number of $K_{1,2,n}$, and then describe all crossing-minimal drawings. 
	In \Cref{sec:K22nProof}, we present our proof of \Cref{th:K22n},
	yielding the tripartite-circle crossing number of $K_{2,2,n}$. We conclude in \Cref{sec:open} with a discussion and open problems.
	
	\section{Tools for counting the number of  crossings}\label{sec:2}
	
	Our tools for counting the number of crossings build on our previous work in \cite{tripartite_JGT}. For a self-contained presentation, we give a concise summary.

	Drawings that minimize the number of crossings are known to be \emph{simple drawings}, that is, drawings where edges are simple curves and any two edges have at most one point in common, which is either a vertex or a crossing. In a tripartite-circle drawing of $K_{m,n,p}$, we label the three circles $\textsc{m}$, $\textsc{n}$, and ~$\textsc{p}$, and their numbers of vertices are $m$, $n$, and $p$, respectively. For an illustration consider \Cref{fig:MNP}.
	
	\begin{figure}[htb]
		\centering
		\includegraphics[page=1]{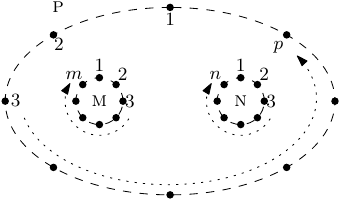}
		\caption{The vertices on the circles $\textsc{m}$ and $\textsc{n}$ are labeled clockwise; the vertices of the circle~$\textsc{p}$ are labeled counterclockwise.}
		\label{fig:MNP}
	\end{figure}
	
	By using projective transformations and the fact that in a tripartite-circle drawing of a complete tripartite graph the three underlying circles cannot be pairwise nested (they are not to be crossed by any edges), we may consider only the drawings where the \emph{outer} circle $\textsc{p}$ contains the \emph{inner} circles $\textsc{m}$ and $\textsc{n}$, see \Cref{fig:MNP}. In such a drawing, we label the vertices on circles $\textsc{m}$ and $\textsc{n}$ in clockwise order and the vertices on circle $\textsc{p}$ in counterclockwise order.  Likewise, we read arcs of circles in clockwise order for inner circles and in counterclockwise order for outer circles.
	
	\subsection{Defining the $x$- and $y$-labels}\label{xiyi}
	As progressively defined in \cite{RT}, \cite{AFS}, and \cite{tripartite_JGT}, we use the following labels. Let $\textsc{a}$, $\textsc{b}$, and $\textsc{c}$ be the three circles and $i$ be a vertex on $\textsc{a}$. The star formed by all edges from $i$ to $\textsc{b}$ together with circle $\textsc{b}$ partitions the plane into disjoint regions, shown in \Cref{fig:defyiLP}. Exactly one of these  regions contains circle $\textsc{a}$. This region is enclosed by two edges from $i$ to $\textsc{b}$ and an arc on $\textsc{b}$ between two consecutive vertices. We define $\vtx{x}{i}{a}{b}$ as the second of these vertices (in clockwise or counterclockwise order depending on whether $\textsc{b}$ is an inner or outer circle, respectively). Similarly, there is exactly one region defined by the star from $i$ that contains the third circle $\textsc{c}$, and the second vertex along circle $\textsc{b}$ (clockwise or counterclockwise as before) on the boundary of this region is $\vtx{y}{i}{a}{b}$. 
	If the two circles are clear from the context, we may also write $x_i$ or $y_i$.
	
	\begin{figure}[htb]
		\centering
		\begin{subfigure}[b]{.32\textwidth}
			\centering
			\includegraphics[page=2]{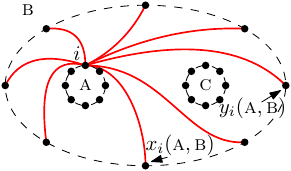}
			\subcaption{Nested: $\textsc a$ is inside $\textsc b$}
		\end{subfigure}\hfill
		\begin{subfigure}[b]{.32\textwidth}
			\centering
			\includegraphics[page=3]{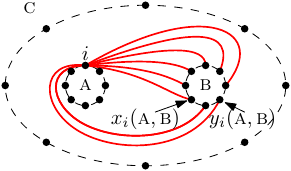}
			\subcaption{Non-Nested: $\textsc a$ is beside $\textsc b$}
		\end{subfigure}
		\hfill
		\begin{subfigure}[b]{.32\textwidth}
			\centering
			\includegraphics[page=4]{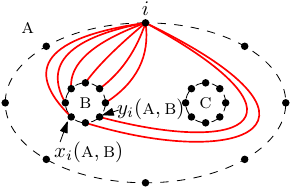}
			\subcaption{Nested: $\textsc b$ is inside  $\textsc a$}
			
		\end{subfigure}
		\caption{The definitions of $\vtx{x}{i}{a}{b}$ and $\vtx{y}{i}{a}{b}$ illustrated for three cases.}
		\label{fig:defyiLP}
	\end{figure}

	\subsection{Counting the crossings}
	
	We introduce notation necessary to state Theorem 7 in \cite{tripartite_JGT} that provides a precise formula for the number of crossings in a tripartite-circle drawing of $K_{m,n,p}$. For vertices $k$ and $\ell$ on a circle with $n$ vertices numbered $1,\ldots, n$ clockwise (respectively, counterclockwise), let
	\begin{equation*}
		d_n(k,\ell):=\ell-k\mod n
	\end{equation*}
	denote the \emph{distance} from $k$ to $\ell$ in clockwise (respectively, counterclockwise) order on the circle. For any $u,v \in \lbrace 1,2,\ldots,n\rbrace$, define
	\[f_n(u,v) := \binom{d_n(u,v)}{2} + \binom{n-d_n(u,v)}{2}.\]
	Throughout the paper, we use the facts that for all $u$ and $v$ we have $f_n(u,v)=f_n(v,u)$ and
	\begin{equation}\label{eq:k22n3}
		f_n(u,v)\geq \binom{\left\lfloor\frac{n}{2}\right\rfloor}{2}
		+\binom{\left\lceil\frac{n}{2}\right\rceil}{2}=\bigg\lfloor\frac{n}{2}\bigg\rfloor\left\lfloor\frac{n-1}{2}\right\rfloor.
	\end{equation}
	As in \cite{tripartite_JGT}, for vertices $i$ and $j$ on an inner (respectively, outer) circle $\textsc{a}$, we write $[i,j]$ for the arc of $\textsc{a}$ read clockwise (respectively, counterclockwise) from $i$ to $j$. We define $[i,j)$, $(i,j]$, and $(i,j)$ similarly, where a square bracket denotes inclusion of the endpoint, and a parenthesis denotes exclusion of the endpoint.
	
	A \emph{cyclic assignment} of \textsc{(a,b,c)} to \textsc{(m,n,p)} is one triple of the set $\mathfrak t:= \{\textsc{(m,n,p)}, \textsc{(n,p,m)}, \textsc{(p,m,n)}\}$. Let the numbers of vertices on the circles \textsc{a, b}, and \textsc{c} be denoted by $a,~b$, and $c$, respectively. The following theorem counts the total number of crossings. 
	
	\begin{theorem}[Theorem 7, \cite{tripartite_JGT}]\label{th:yitheorem}
		The number of crossings in a simple tripartite-circle drawing of~ $K_{m,n,p}$ is given by 
		\[
		\sum_{\textsc{(a,b,c)}\in \mathfrak t} \left(
		\sum_{\substack{i < j\\i,j \in \textsc{A}}} f_b\big(\vtx{x}{i}{a}{b},\vtx{x}{j}{a}{b}\big)
		+\sum_{\substack{i \in \textsc{A}\\j \in \textsc{B}}} f_c\big(\vtx{y}{i}{a}{c},\vtx{y}{j}{b}{c}\big)
		\right).
		\]
	\end{theorem}
	
	In this result (see Lemmas 5 and 6 in \cite{tripartite_JGT}), $\sum_{\substack{i < j\\i,j \in \textsc{A}}} f_b\big(\vtx{x}{i}{a}{b},\vtx{x}{j}{a}{b}\big)$ counts the number of crossings determined by edges incident to one vertex on circle $\textsc{A}$ and another on circle $\textsc{B}$; whereas  $\sum_{\substack{i \in \textsc{A}\\j \in \textsc{B}}} f_c\big(\vtx{y}{i}{a}{c},\vtx{y}{j}{b}{c}\big)$ counts the number of crossings determined by an edge between circles $\textsc{A}$ and  $\textsc{C}$ with an edge between circles $\textsc{B}$ and  $\textsc{C}$.

	\section{Proofs of Proposition \ref{obs:K11n} and \Cref{thm:K12n}}\label{sec:K12n}
	
	In this section, we determine the exact values of $\crN{3}(K_{1,1,n})$ and $\crN{3}(K_{1,2,n})$.

	\KOneOneN*
	
	\begin{proof}
		Let $D$ be any crossing-minimal drawing of $K_{1,1,n}$. By minimality, $D$ is simple; in particular, any two edges share at most one point. 
		Let $v$ and $w$ be the vertices in the $1$-vertex parts of $K_{1,1,n}$, and let $e = \{v,w\}$ denote their edge. Because any other edge $e'$ of $K_{1,1,n}$ is incident to  either $v$ or $w$, $e'$ does not cross $e$ in $D$; otherwise $e$ and $e'$ would share two points. Consequently, $e$ is not involved in any crossing and we can easily replace $e$ with two parallel curves along $e$ from $v$ to $w$ (representing a circle). 
		This yields a bipartite-circle drawing of $K_{2,n}$ with the same number of crossings. By the reverse substitution, any bipartite-circle drawing of~$K_{2,n}$ yields a tripartite-circle drawing of $K_{1,1,n}$ with the same number of crossings. Thus the two crossing numbers are equal. Finally, $\crN{2} (K_{2,n})=\binom{n}{2}-\nicefrac{1}{12}(4n^2-4-n^2+\gcd(2,n)^2)=\lceil \nicefrac{1}{4}\cdot n(n-2)\rceil$ by \Cref{eq:ccr_Kmn}.
	\end{proof}
	
	It turns out that $\crN{3} (K_{1,2,n})\approx \nicefrac{3}{4}\cdot n^2$, which is smaller than $\crN{2} (K_{3,n})\approx \nicefrac{5}{6}\cdot n^2$. 
	
	\KOneTwoN*
	
	The proof of \Cref{thm:K12n} is a direct consequence of the following three lemmas: \Cref{lem:small cases} takes care of the cases $2 \le n \le 4$, and \Cref{lem:upperbound,lem:lowerbound} settle the upper and lower bounds for $n\geq 5$, respectively.
	\begin{lemma}\label{lem:small cases}
		It holds that
		$\crN{3} (K_{1,2,2})=1,\crN{3} (K_{1,2,3})=3$, and $\crN{3} (K_{1,2,4})=8$.
	\end{lemma}
	\begin{proof}
		For $n=2,3,4$, \cref{fig:K12nSmall} presents tripartite-circle drawings of $K_{1,2,n}$ with $1$, $3$, and $8$ crossings, respectively, which prove the upper bounds.
		\begin{figure}[htb]
			\centering	
			\begin{subfigure}[c]{.3\textwidth}
				\centering
				\includegraphics[page=5]{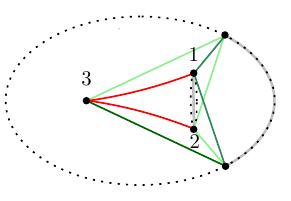}
				\subcaption{$K_{1,2,2}$}
				\label{fig:K12nSmallA}
			\end{subfigure}\hfill
			\begin{subfigure}[c]{.3\textwidth}
				\centering
				\includegraphics[page=6]{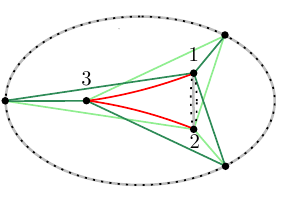}
				\subcaption{$K_{1,2,3}$}
				\label{fig:K12nSmallB}
			\end{subfigure}\hfill
			\begin{subfigure}[c]{.3\textwidth}
				\centering
				\includegraphics[page=7]{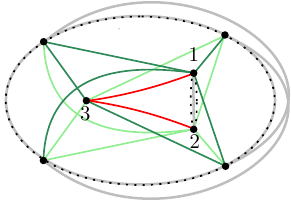}
				\subcaption{$K_{1,2,4}$}
				\label{fig:K12nSmallC}
			\end{subfigure}\hfill
			\caption{Crossing-minimal tripartite-circle drawings of $K_{1,2,n}$ for $n=2,3,4$; crossing-optimal drawings of the complete graph $K_{n+3}$ are obtained by adding gray edges.
			} 
			\label{fig:K12nSmall}
		\end{figure}
		
		For the lower bounds, we use the known crossing numbers $\crg(K_5)=1$, $\crg(K_6)=3$, and $\crg(K_7)=9$.  For any fixed drawing $D$, we denote its number of crossings by $\crg(D)$. For $n = 2, 3, 4$, let $D'$ be a crossing-minimal tripartite-circle drawing of $K_{1,2,n}$. We construct a drawing $D$ of $K_{n+3}$ as an extension of $D'$. For $n=2,3$, it suffices to add edges along the arcs of the $n$-vertex circle and the $2$-vertex circle as depicted by the gray edges in \cref{fig:K12nSmallA,fig:K12nSmallB}, introducing no new crossings. Thus $\crg(K_{n+3}) \le \crg(D)=\crg(D') = \crN{3}(K_{1,2,n})$, so $\crN{3} (K_{1,2,2})=1$ and $\crN{3} (K_{1,2,3})=3$. For $n=4$, we obtain a drawing $D$ of $K_7$ extending $D'$, shown in \cref{fig:K12nSmallC}, by adding one edge along a semicircle of the $2$-vertex circle, four edges along the arcs of the $4$-vertex circle, and two additional edges which cross each other but no other edges. Then $9 = \crg(K_7) \le \crg(D)=\crg(D')+1 = \crN{3}(K_{1,2,4})+1$, so $\crN{3}(K_{1,2,4})=8$.
	\end{proof}
	
	For the rest of this section, we label the point on the $1$-vertex circle \textsc{m} by $3$, and the points of the $2$-vertex circle \textsc{n} by $1$ and $2$, in such a way that the interior of triangle $\{1,2,3\}$ (read clockwise) does not contain the $n$-vertex circle~\textsc{p} (see  Figures \ref{fig:K12n_segments} and \ref{fig:K12n_segmentsCount}).
	
	\begin{lemma}\label{lem:upperbound}
		For every integer $n\geq 5$, $\crN{3} (K_{1,2,n})\leq\lfloor\nicefrac{3}{4}\cdot n^2\rfloor-n$.
	\end{lemma}
	\begin{proof}
		The family of straight-line drawings in \Cref{fig:K12n_segments} yields the upper bound having exactly
		$\lfloor\nicefrac{3}{4}\cdot n^2\rfloor-n$ crossings. 
		In the drawing, we group all vertices on circle \textsc{p} except $w$ 
		into four sets $A,B,C$, and $D$ whose sizes depend on the parameter $t$; see \Cref{fig:K12n_segments}. (Although a single drawing would be enough to settle the bound, this family of drawings will be part of the full classification of crossing-optimal drawings in Section 3.1.) 
		\begin{figure}[htb]
			\centering	\includegraphics[page=8]{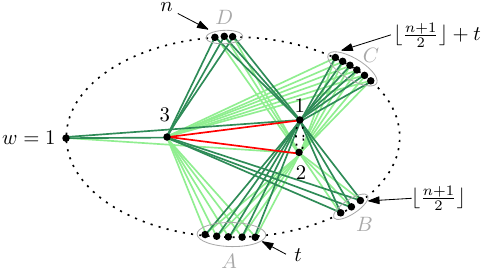}
			\caption{A crossing-minimal tripartite-circle drawing of $K_{1,2,n}$ for all $n\geq 4$ and $2\leq t \leq \lceil(n-1)/2\rceil$.} 
			\label{fig:K12n_segments}
		\end{figure}
		
		We color the edges 13 and 23 red; all other edges are green. There are $|A\cup D|$ crossings with red edges. The green-green crossings are determined by two vertices in $\{1,2,3\}$ and two vertices on~\textsc{p}: more precisely, by 1, 2, and any two vertices in $A\cup D\cup\{w\}$ or in $B\cup C$; by 1, 3, and any two vertices in $C\cup D\cup\{w\}$ or in $A\cup B$; or by 2, 3 and any two vertices in $A\cup B\cup\{w\}$ or in $C\cup D$. Since $|A\cup D\cup\{w\}|=|C\cup D|=\lceil(n+1)/2\rceil-1$ and $|A\cup B\cup\{w\}|=|B\cup C|=\lfloor(n+1)/2\rfloor$, the total number of crossings is 
		\[
		\left\lceil \frac{n+1}{2}\right\rceil-2
		+2\binom{\lceil \frac{n+1}{2}\rceil-1}{2}
		+2\binom{\lfloor \frac{n+1}{2}\rfloor}{2}
		+\binom{\lfloor \frac{n+1}{2} \rfloor-1}{2}
		+\binom{\lceil \frac{n+1}{2}\rceil}{2}\\
		=\big\lfloor\nicefrac{3}{4}\cdot n^2\big\rfloor-n.\qedhere\]
	\end{proof}
	\begin{lemma}\label{lem:lowerbound}
		For every integer $n\geq 5$, $\crN{3} (K_{1,2,n})\ge \lfloor\nicefrac{3}{4}\cdot n^2\rfloor-n$.
	\end{lemma}
	\begin{proof}
		Start with an arbitrary simple tripartite-circle drawing of $K_{1,2,n}$; for an example consider \cref{fig:K12n_segmentsCount}. 
		For $i=1,2$, let $U_i$ be the set of vertices $u$ on circle~\textsc{p} such that $iu$ crosses $j3$, where $\{i,j\}=\{1,2\}$.
		Note that $U_1=[y_1,x_1)$ and $U_2=[x_2,y_2)$ and so the number of crossings with the edges $13$ and $23$ is $d_n(y_1,x_1)+d_n(x_2,y_2)$.
		By \Cref{th:yitheorem}, there are $f_n(x_1,x_2)$ crossings of edges $1u$ and $2v$, $f_n(y_1,y_3)$ crossings of edges $1u$ and $3v$, and $f_n(y_2,y_3)$ crossings of edges $2u$ and $3v$, where $u$ and $v$ are vertices on circle~\textsc{p}. 
		
		Therefore, the total number of crossings is
		\begin{equation}\label{eq:k12n1}
			d_n(y_1,x_1)+d_n(x_2,y_2)+f_n(x_1,x_2)+f_n(y_1,y_3)+f_n(y_2,y_3).
		\end{equation}
		In order to simplify the presentation, we write $d_x:=d_n(x_1,x_2)$ and $d_y:=d_n(y_2,y_1)$.

		We show that for fixed $y_1$ and $y_2$, the value of $f_n(y_1,y_3)+f_n(y_2,y_3)$ is minimized when $y_3$ divides the longer of the intervals $[y_2,y_1)$ and $[y_1,y_2)$ in half. For $\{i,j\} = \{1,2\}$ and $y_3\in(y_i,y_j]$, we write $D:= d_n(y_3,y_j)$ and $c_i:=d_n(y_j,y_i)$. Using the identity $\binom{x}{2}+\binom{n-x}{2} = \binom{n}{2} - x(n-x)$, we obtain
		\begin{align*}
			f_n(y_1, y_3) + f_n(y_2,y_3) &= 2\binom{n}{2} - (n-D)D - (n-c_i-D)(c_i+D)\\
			&= 2D^2 + 2(c_i-n)D + c_i(c_i-n) + 2\binom{n}{2}.
		\end{align*}
		This is a quadratic function in $D$ and thus minimized when $D \in \{\floor{\frac{n-c_i}{2}}, \ceil{\frac{n-c_i}{2}}\}$, so
		\begin{equation}\label{eq:ineq_2fs}
			f_n(y_1,y_3)+f_n(y_2,y_3)\geq 2\binom{n}{2}-\left\lfloor\dfrac{n^2-\min(d_y,n-d_y)^2}{2}\right\rfloor.
		\end{equation}

		Since the drawing is simple, and the edges $1u$ and $2v$ cross for any $u\in U_1$ and $v\in U_2$ (see \Cref{fig:K12n_segmentsCount}), the sets $U_1$ and $U_2$ must be disjoint.  If $U_i=\emptyset$, then $x_i=y_i$. If $U_1\neq\emptyset$, then its first and last points counterclockwise along the circle~\textsc{p} are $y_1$ and $x_1-1$, respectively. Similarly, if $U_2\neq\emptyset$, then its first and last points  counterclockwise along the circle~\textsc{p} are $x_2$ and $y_2-1$, respectively. We consider three cases.
		
		\begin{figure}[htb]
			\centering	\includegraphics[page=9]{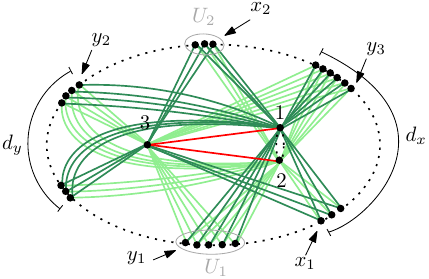}
			\caption{A simple tripartite-circle drawing of $K_{1,2,n}$ to illustrate \Cref{lem:lowerbound}.
			} 
			\label{fig:K12n_segmentsCount}
		\end{figure}

		\begin{case}\label{case:1_K12n}
			$U_1$ and $U_2$ are both empty. Then $x_1=y_1$ and $x_2=y_2$, which implies $d_n(y_1,x_1)+d_n(x_2,y_2)=0$ and $d_x+d_y=n$. Let $d=\min\{d_x,d_y\}$. For $n\ge 5$, using \Cref{eq:ineq_2fs}, Expression~(\ref{eq:k12n1}) becomes
			\begin{align*}
				&f_n(x_1,x_2)+f_n(y_1,y_3)+f_n(y_2,y_3)
				\geq 3\binom{n}{2}-d(n-d)-\bigg\lfloor \frac{n^2-d^2}{2}\bigg\rfloor\\
				&= 3\binom{n}{2}-\bigg\lfloor \frac{(n-d)(n+3d)}{2}\bigg\rfloor
				\geq  3\binom{n}{2}-\big\lfloor\nicefrac{2}{3}\cdot n^2\big\rfloor
				> \big\lfloor\nicefrac{3}{4}\cdot n^2\big\rfloor-n.
			\end{align*}
			Therefore, this case does not yield crossing-minimal drawings.
		\end{case}
		
		It remains to consider the case when at least one of the sets $U_1$ or $U_2$ is nonempty. By symmetry, we assume that $U_1$ is nonempty and analyze two cases.
		\begin{case} \label{case:2_K12n}
			$U_1$ is nonempty and $x_2\in [x_1,y_1]$. If $U_2$ is nonempty, since $U_1$ and $U_2$ are disjoint, the vertices $y_1,x_1,x_2,y_2$ appear in this order counterclockwise along the $n$-vertex circle (some of them could overlap). If $U_2=\emptyset$, then $x_2=y_2$ and so $y_1,x_1,x_2=y_2$ appear in this order counterclockwise along the $n$-vertex circle. 
			In both cases,  $$d_n(y_1,x_1)+d_x+d_n(x_2,y_2)+d_y=n.$$ Since $U_1\neq\emptyset$, then  $d_n(y_1,x_1)\geq 1$ and so $d_x+d_y\leq n-1$. Hence, Expression (\ref{eq:k12n1}) becomes
			\begin{align}\label{eq:k12n2}
				&n-d_x-d_y+f_n(x_1,x_2)+f_n(y_1,y_3)+f_n(y_2,y_3)\nonumber\\
				&\geq n-d_x-d_y+3\binom{n}{2}-d_x(n-d_x)
				-\left\lfloor\dfrac{n^2-\min(d_y,n-d_y)^2}{2}\right\rfloor\nonumber\\
				&= 3\binom{n}{2}-d_x\left(n+1-d_x\right)-\left\lfloor
				\begin{cases}
					(n-d_y)(n-2+d_y)/2&\text{ if }d_y\leq n/2\\
					(n^2-(n-d_y)(n-d_y+2))/2&\text{ if }d_y> n/2
				\end{cases}
				\right\rfloor.
			\end{align}
			
			Assume first that $d_y\leq n/2$. 
			For $n\geq 4$, Expression (\ref{eq:k12n2}) 
			is minimized when $d_x=\lfloor \frac{n+1}{2} \rfloor$ and $d_y=1$ implying
			\begin{align*}
				&n-d_x-d_y+f_n(x_1,x_2)+f_n(y_1,y_3)+f_n(y_2,y_3)\nonumber\\
				&\geq 
				3\tbinom{n}{2}
				-\left\lfloor\tfrac{n+1}{2}\right\rfloor
				\left\lceil\tfrac{n+1}{2}\right\rceil
				-\left\lfloor\nicefrac{1}{2}\cdot(n-1)^2\right\rfloor
				\geq\lfloor\nicefrac{3}{4}\cdot n^2\rfloor-n.
			\end{align*}

			Assume now that $d_y> n/2$. Then
			\begin{align}\label{eq:k12n4}
				&n-d_x-d_y+f_n(x_1,x_2)+f_n(y_1,y_3)+f_n(y_2,y_3)\nonumber\\
				&\geq 
				3\tbinom{n}{2}
				-d_x\left(n+1-d_x\right)
				-\nicefrac{1}{2}\cdot(n^2-(n-d_y)(n-d_y+2)).
			\end{align}
			When $n\geq 7$, Expression (\ref{eq:k12n4}) is minimized when $d_x = \frac{n-1}{3}$ and $d_y = \frac{2(n-1)}{3}$, and the inequality becomes
			\begin{align*}
				&n-d_x-d_y+f_n(x_1,x_2)+f_n(y_1,y_3)+f_n(y_2,y_3)\nonumber\\
				&\geq 
				3\tbinom{n}{2}
				-\tfrac{n-1}{3}\left(n+1-\tfrac{n-1}{3}\right)
				-\tfrac{1}{2}\left(n^2-\left(n-\tfrac{2(n-1)}{3}\right)\left(n-\tfrac{2(n-1)}{3}+2\right)\right)\nonumber\\
				&=\nicefrac{1}{6}\cdot(5n^2-7n+8)>\lfloor\nicefrac{3}{4}\cdot n^2\rfloor-n.
			\end{align*}
			Otherwise, when $n = 5$ or $n=6$, Expression (\ref{eq:k12n4}) is minimized when $d_x = \frac{n-3}{2}$ and $d_y = \frac{n+1}{3}$, and the inequality becomes 
			\begin{align*}
				&n-d_x-d_y+f_n(x_1,x_2)+f_n(y_1,y_3)+f_n(y_2,y_3)\nonumber\\
				&\geq 
				3\tbinom{n}{2}
				-\tfrac{n-3}{2}\left(n+1-\tfrac{n-3}{2}\right)
				-\tfrac{1}{2}\left(n^2-\left(n-\tfrac{n+1}{2}\right)\left(n-\tfrac{n+1}{2}+2\right)\right)\nonumber\\
				&=\nicefrac{1}{8}\cdot(7n^2-14n+27)>\lfloor\nicefrac{3}{4}\cdot n^2\rfloor-n.
			\end{align*}
		\end{case}
		
		\begin{case} \label{case:3_K12n}
			$U_1$ is nonempty and $x_2\in (y_1,x_1)$. In this case,
			$U_2=\emptyset$ because $U_1$ and $U_2$ are disjoint. In this case $x_2=y_2$ and so $d_y\geq 1$. Then$$d_n(y_1,x_1)+d_n(y_2,x_2)=d_n(y_1,x_1)=n-d_x+d_y.$$ 
			If $d_y> n/2$, then $d_n(y_1,x_1)+d_n(y_2,x_2)\geq d_y > n/2$. Using  \Cref{eq:k22n3}, we have
			\begin{align*}
				&d_n(y_1,x_1)+d_n(x_2,y_2)+f_n(x_1,x_2)+f_n(y_1,y_3)+f_n(y_2,y_3)\\
				&> \left\lfloor\frac{n}{2}\right\rfloor+3\left\lfloor\frac{n}{2}\right\rfloor\left\lfloor\frac{n-1}{2}\right\rfloor
				\geq \lfloor\nicefrac{3}{4}\cdot n^2\rfloor-n.
			\end{align*}
			If $1\leq d_y\leq n/2$, then $\min(d_y,n-d_y)=d_y$ and so Expression (\ref{eq:k12n1}) can be written as
			\begin{align}\label{eq:k12n5}
				&n+d_y-d_x+f_n(x_1,x_2)+f_n(y_1,y_3)+f_n(y_2,y_3)\nonumber\\
				&\geq n+d_y-d_x+3\binom{n}{2}-d_x(n-d_x)
				-\left\lfloor\dfrac{n^2-d_y^2}{2}\right\rfloor\nonumber\\
				&= 3\binom{n}{2}-d_x\left(n+1-d_x\right)-\left\lfloor
				\frac{(n+d_y)(n-2-d_y)}{2}
				\right\rfloor.
			\end{align}
			For $n\geq 3$, Expression (\ref{eq:k12n5}) is minimized when  $d_x=\lfloor \frac{n+1}{2} \rfloor$ or $\lceil \frac{n+1}{2} \rceil$ and $d_y=1$. Thus
			\begin{align*}
				&n-d_x-d_y+f_n(x_1,x_2)+f_n(y_1,y_3)+f_n(y_2,y_3)\nonumber\\
				&\geq 3\tbinom{n}{2}-\left\lfloor\tfrac{n+1}{2}\right\rfloor
				\left\lceil\tfrac{n+1}{2}\right\rceil
				-\left\lfloor\nicefrac{1}{2}\cdot(n+1)(n-3)\right\rfloor
				>\lfloor\nicefrac{3}{4}\cdot n^2\rfloor-n.\qedhere
			\end{align*}
		\end{case}
	\end{proof}
	
	
	\subsection{List of all crossing-minimal drawings}
	By allowing $n=2,3$, and $4$ in the proof of the lower bound of \Cref{thm:K12n}, it is possible to track all the cases when equality holds. In Case \ref{case:1_K12n}, the strict inequality becomes an identity for $n=2,3,$ and $4$. The other two inequalities are tight when $\min\{d_x,d_y\}=\lfloor n/3\rfloor$ or $\lceil n/3 \rceil$. Figures~\ref{fig:K12n_segments_small}(a)-(e) 
	show the corresponding crossing-minimal drawings of $K_{1,2,n}$. 
	\begin{figure}[hbtp]
		\centering	\includegraphics[page=10]{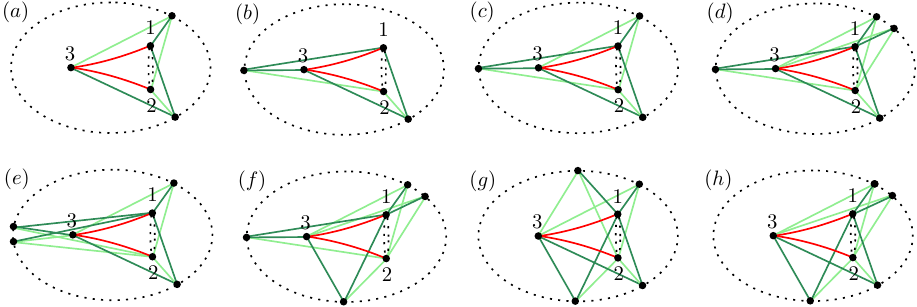}
		\caption{All crossing-minimal tripartite-circle drawings of $K_{1,2,n}$ for $n=2,3$, and  $4$.}
		\label{fig:K12n_segments_small}
	\end{figure}
	
	The other three drawings of $K_{1,2,4}$, depicted in Figures~\ref{fig:K12n_segments_small}(f)-(h),  come from Case \ref{case:2_K12n}. In this case, Inequality (\ref{eq:k12n4}) is still strict for $n=2,3,4$. \Cref{eq:k12n2} for $d_y\leq n/2$ is tight when $y_3$  divides the longest of the intervals $[y_2,y_1]$ or $[y_1,y_2]$ in almost half (so that \Cref{eq:ineq_2fs} is also tight) and
	\begin{equation*}
		(d_x,d_y)=
		\begin{cases}
			(\lfloor\frac{n+1}{2}\rfloor,1), (\lceil\frac{n+1}{2}\rceil,1)&\text{ if }n\geq 5,\\
			(\frac{n}{2},0)&\text{ if }n\geq2 \text{ even},\\
			(\frac{n}{2},2),(\frac{n}{2}+1,2)&\text{ if }n\geq8 \text{ even},\\
			(3,2),(4,0)&\text{ if }n=6,\\
			(2,1),(3,0)&\text{ if }n=4.
		\end{cases}
	\end{equation*}
	Crossing-optimal configurations exist for each of these situations. \Cref{fig:K12n_segments} (with $A=U_1, B=V_1,C=V_2,D=U_2$) shows the case $(d_x,d_y)=(\lfloor(n+1)/2\rfloor,1)$ for $n\geq 4$ (and $y_3$ slightly closer to $y_1$ than $y_2$ when $n$ is even). 
	\Cref{fig:K12n_segments2} shows the case  $(d_x,d_y)=(n/2,0)$ for $n\geq 2$ even;
	note that these two figures in the particular case when $n=4$ are shown in Figures~\ref{fig:K12n_segments_small}(f) and \ref{fig:K12n_segments_small}(g). \cref{fig:K12n_segments_small}(h) corresponds to the case $(d_x,d_y)=(3,0)$ and $n=4$.
	
	\begin{figure}[htbp]
		\centering
		\includegraphics[page=11]{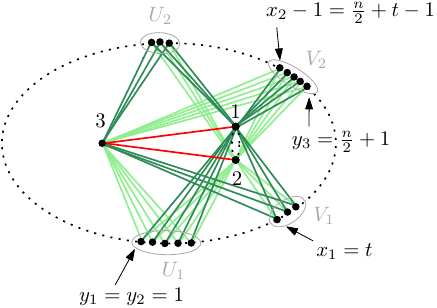}
		\caption{A crossing-minimal tripartite-circle drawing of $K_{1,2,n}$ for even $n\geq 4$  and $2\leq t \leq n/2+1$. 
		} 
		\label{fig:K12n_segments2}
	\end{figure}
	
	Finally, Case \ref{case:3_K12n} requires $n\geq 3$, and the last inequality is still strict for $n=3$ and $4$. Hence, there are no crossing-minimal drawings in this case. 
	
	\section{Proof of \Cref{th:K22n}}\label{sec:K22nProof}
	
	In this section, we determine the tripartite-circle crossing number of $K_{2,2,n}$. 
	
	\KTwoTwoN*

	The proof of \Cref{th:K22n} is a direct consequence of the results in  \Cref{sec:UB,sec:LB}. First, in \Cref{sec:UB}, we give a construction to show the upper bound. Then, in \Cref{sec:LB}, we show the lower bound by analyzing several cases to show that our construction actually minimizes the number of crossings of $K_{2,2,n}$. 
	
	\subsection{Upper bound of \Cref{th:K22n}}
	\label{sec:UB}
	
	The following lemma states the upper bound of \Cref{th:K22n}.
	\begin{lemma}\label{lem:22nUB}
		For every integer $n\geq 3$,
		\[
		\crN{3}(K_{2,2,n}) \le
		6\bigg\lfloor\frac{n}{2}\bigg\rfloor\left\lfloor\frac{n-1}{2}\right\rfloor+2n-3.
		\]
	\end{lemma}
	
	\begin{proof} 
		In \Cref{fig:K22n_segments}, we define two drawings that achieve the bound for odd values of $n\geq 3$ in drawing (a) and for all values of $n\geq 4$ in drawing (b). Both drawings have \textsc{m} and \textsc{n} as the inner circles with two vertices each. We color the edges between $\textsc{m}$ and $\textsc{n}$ red and all others green.
		We divide the vertices on $\textsc{p}$ into groups which are called $A,B,C,D$, where groups~$C,D$ are empty in drawing (a). All green edges are straight line segments except for $1x_1$ and $2x_2$ in (a) and $1y_2$, $2x_3$, $3y_4$, and $4x_1$ in (b). If these edges are replaced by straight line segments, the number of crossings is easy to determine.
		These replacements add two crossings to (a) and four to (b).  For simplicity, we define $c_n:=6\left(\binom{\floor{n/2}}{2}+\binom{\lceil{n/2}\rceil}{2}\right)=6\lfloor\frac{n}{2}\rfloor\left\lfloor\frac{n-1}{2}\right\rfloor$.
		
		\begin{figure}[htb]
			\centering		
			\begin{subfigure}[c]{.45\textwidth}
				\centering
				\includegraphics[page=12]{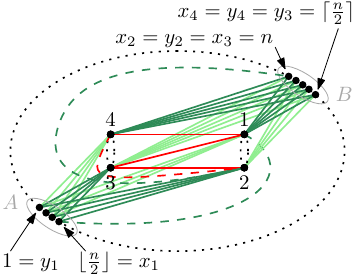}
				\subcaption{$n\geq 3$ odd }
				\label{fig:K22n_segmentsA}
			\end{subfigure}\hfill
			\begin{subfigure}[c]{.45\textwidth}
				\centering
				\includegraphics[page=13]{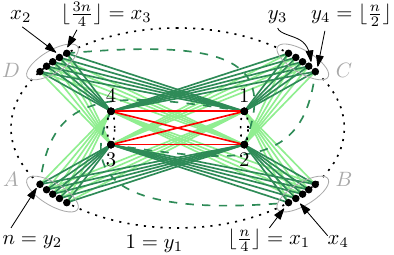}
				\subcaption{$n\geq 4$}
				\label{fig:K22n_segmentsB}
			\end{subfigure}
			\caption{Two crossing-minimal tripartite-circle drawings of $K_{2,2,n}$, where all but at most four edges are straight line segments.
			}
			\label{fig:K22n_segments}
		\end{figure}

		After replacing these (curved) green edges in \Cref{fig:K22n_segments}(a) by straight line segments, the drawing has $3\floor{n/2}+\ceil{n/2}=2n-1$ red-green crossings (when $n$ is odd) and no red-red crossing. In order to count the green-green crossings, note that every crossing is determined by any two vertices on the inner circles and any two vertices in $A$ or any two vertices in $B$. Thus there are $c_n$ green-green crossings and the number of crossings in the drawing displayed in \Cref{fig:K22n_segments}(a) is
		$c_n+2n-1-2=6\lfloor\frac{n}{2}\rfloor\left\lfloor\frac{n-1}{2}\right\rfloor+2n-3.$
		
		Similarly, after the replacement in~\Cref{fig:K22n_segments}(b), the drawing has $2n$ red-green crossings and $1$ red-red crossing. The green-green crossings are determined by either two vertices on the same inner circle and any two vertices in $A\cup D$ or any two vertices on $B\cup C$; or by two vertices on different inner circles and any two vertices in $A\cup B$ or any two vertices on $C\cup D$. Note that $\{|A\cup D|,|B\cup C|\}=\{|A\cup B |,|C\cup D|\}=\{\floor{n/2},\lceil{n/2}\rceil\}$ and thus there are $c_n$ green-green crossings. The number of crossings in the drawing displayed in \Cref{fig:K22n_segments}(b) is
		$c_n+2n+1-4=6\lfloor\frac{n}{2}\rfloor\left\lfloor\frac{n-1}{2}\right\rfloor+2n-3.$
	\end{proof}
	
	\subsection{Lower bound of \Cref{th:K22n}}\label{sec:LB}
	
	Now, we turn our attention towards proving the following lower bound, i.e., we want to show that the drawings given in \Cref{sec:UB} have the minimum number of crossings among tripartite-circle drawings.
	
	\begin{restatable}{lemma}{LowerBound}\label{lem:22nLB}
		For every integer $n\geq 3$,
		\[
		\crN{3}(K_{2,2,n}) \geq
		6\bigg\lfloor\frac{n}{2}\bigg\rfloor\left\lfloor\frac{n-1}{2}\right\rfloor+2n-3.
		\]
	\end{restatable}
	
	In order to prove \Cref{lem:22nLB}, we need a series of lemmas. \Cref{le:k22n} categorizes all simple tripartite-circle drawings into  Types 1, 2, 3, or 4. \Cref{lem:redgreencrossings} provides expressions for the exact number of crossings in each of the four types; Lemmas \ref{lem:redgreencrossings}, \ref{lem:mixed}, and \ref{lem:ysNEW} give bounds for different components of  the expressions in \Cref{lem:redgreencrossings}; \Cref{lem:type234} settles the lower bound for the drawings of Types 2, 3, 4 and \Cref{lem:type1} for the drawings of Type 1.

	It is sufficient to consider simple tripartite-circle drawings. To analyze all such drawings, we partition the set of simple drawings of $K_{2,2,n}$ by the induced subdrawings of $K_{2,2,0}$ depicted in \Cref{fig:k220drawings}.
	
	\begin{figure}[htbp]
		\centering
		\begin{subfigure}[b]{.25\textwidth}
			\centering
			\includegraphics[page=14]{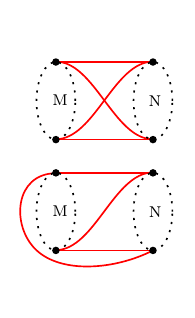}
			\caption{}
			\label{fig:k22drawings}
		\end{subfigure}
		\hfill
		\begin{subfigure}[b]{.7\textwidth}
			\centering
			\includegraphics[page=15]{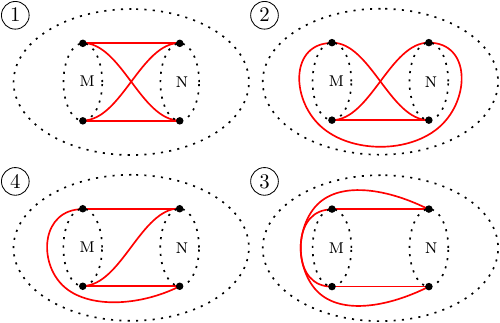}
			\caption{}
			\label{fig:k220drawings}
		\end{subfigure}
		\caption{(a) The two simple bipartite-circle drawings of $K_{2,2}$. (b) The four simple tripartite-circle drawings of $K_{2,2,0}$.}
	\end{figure}
	
	\begin{lemma}\label[lemma]{le:k22n}
		Up to topological equivalence, there are exactly four simple tripartite-circle drawings of $K_{2,2,0}$. 
	\end{lemma}
	
	\begin{proof} Up to topological equivalence, there are two different simple bipartite-circle drawings of $K_{2,2}$ on the sphere, namely the ones depicted in \Cref{fig:k22drawings}. 
		The edges in the first drawing define five regions, and the edges in the second drawing define four regions. 
		
		Placing the third circle in each of these regions, then finding the equivalent drawing such that the third circle encloses \textsc{m} and \textsc{n}, yields four distinct drawings, shown in \Cref{fig:k220drawings}. For example, if we place the third circle in the upper triangular region of the first drawing in \Cref{fig:k22drawings}, we then need to reroute the upper horizontal edge around \textsc{m} and \textsc{n} and obtain drawing \textcircled{{\scriptsize 2}} in \Cref{fig:k220drawings}.
	\end{proof}
	
	\paragraph{Drawings of type $i\in\{1,2,3,4\}$.} By \Cref{le:k22n}, any simple tripartite-circle drawing of $K_{2,2,n}$ can be seen as an extension of one of the four drawings of $K_{2,2,0}$ in \Cref{fig:k220drawings}. For example, Figure 3(a) displays a drawing of type~4
	and Figure 3(b) a drawing of type~1. In our figures, we color the edges of $K_{2,2,0}$ red and the remaining $4n$ edges green. For an illustration consider \Cref{fig:K22nFormula}; note that some edges are omitted for clarity. We first count the number of crossings with red edges (i.e., the red-red and red-green crossings).
	
	\begin{lemma}\label{lem:redgreencrossings}
		In a simple tripartite-circle drawing $D$ of $K_{2,2,n}$ with $n\geq3$, the number of red-red and red-green crossings is at least
		\begin{equation}\label{eq:red}
			\begin{cases} 2\cdot(d_n(y_1,x_1)+d_n(x_2,y_2)+d_n(y_3,x_3)+d_n(x_4,y_4))+1 &\mbox{if $D$ is of type 1, } \\
				2n+1 &\mbox{if $D$ is of type 2 or 3, } \\
				2\cdot(d_n(y_1,x_1)+d_n(x_2,y_2))+n &\mbox{if $D$ is of type 4.}
			\end{cases}
		\end{equation}
		The number of green-green crossings is \begin{equation}
			\label{eq:k22n1}
			f_n(x_1,x_2)+f_n(x_3,x_4)+ f_n(y_1,y_3)+f_n(y_1,y_4)+f_n(y_2,y_3)+f_n(y_2,y_4),
		\end{equation} 
		and the total number of crossings is the sum of (\ref{eq:red}) and (\ref{eq:k22n1}).
	\end{lemma}
	
	\begin{proof}
		\begin{figure}[htbp]
			\centering
			\includegraphics[page=16]{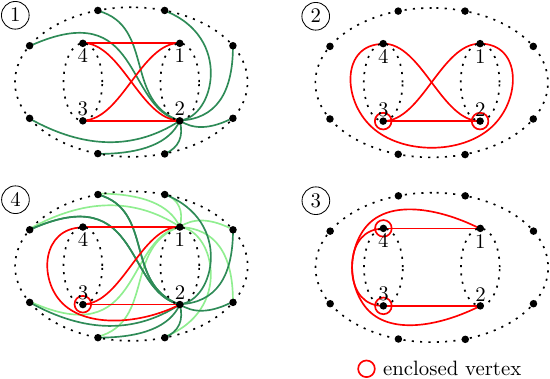}
			\caption{The four simple tripartite-circle drawings of $K_{2,2,n}$. Red edges connect vertices on the two inner circles; all other edges are green (some of them omitted for clarity). Drawings $2$ and $3$ have 
				two enclosed vertices. A green edge from these vertices must cross a red edge.}
			\label{fig:K22nFormula}
		\end{figure}
		
		There is exactly one red-red crossing for types 1, 2, and 3 and none for type~4. \emph{Enclosed} vertices in drawings of types 2, 3, and 4 are those separated from the outer circle by red edges. The green edges incident to enclosed vertices must cross at least one red edge, so each enclosed vertex contributes at least $n$ to the number of red-green crossings. For drawings of types $2$ and $3$, the two enclosed vertices guarantee at least $2n$ red-green crossings, and adding the $1$ red-red crossing yields the claimed count.
		
		Recall that we consider the vertices on the outer circle $P$ in counterclockwise order. In a drawing of type 1, a green edge from vertex $i\in \{1,3\}$ crosses two red edges if the other vertex lies in the interval $[y_i,x_i)$; otherwise it does not cross any red edges. Note that the number of these vertices is $d_n(y_i,x_i)$. Likewise, a green edge from a vertex $i\in \{2,4\}$ crosses two red edges if the other vertex lies in the interval $[x_i,y_i)$. For type $1$ drawings therefore the number of red-green crossings is at least $2\cdot(d_n(y_1,x_1)+d_n(x_2,y_2)+d_n(y_3,x_3)+d_n(x_4,y_4))$. Adding the $1$ red-red crossing yields the claimed count.
		
		The same holds for green edges incident to vertices 1 or 2 in a drawing of type 4, so the number of red-green crossings in a type $4$ drawing is at least $2\cdot(d_n(y_1,x_1)+d_n(x_2,y_2))+n$, where the $n$ term counts red-green crossings in which the green edge is incident to the enclosed vertex $3$.
		
		We use \Cref{th:yitheorem} (and the note immediately following it) to count the green-green crossings as 
		\begin{equation*}
			f_n(x_1,x_2)+f_n(x_3,x_4)+ f_n(y_1,y_3)+f_n(y_1,y_4)+f_n(y_2,y_3)+f_n(y_2,y_4).
			\qedhere
		\end{equation*} 
	\end{proof}
	
	Now to complete the proof of \Cref{lem:22nLB} and therefore of \Cref{th:K22n} it is sufficient to show that the sum of (\ref{eq:red}) and (\ref{eq:k22n1}) is at least  
	$6\lfloor\frac{n}{2}\rfloor\left\lfloor\frac{n-1}{2}\right\rfloor+2n-3$ for any of the four types of drawings. 
	By \Cref{eq:k22n3}, each term of \eqref{eq:k22n1} satisfies $f_n(a,b)\geq \big\lfloor\frac{n}{2}\big\rfloor\left\lfloor\frac{n-1}{2}\right\rfloor$.
	
	We start with two lemmas bounding some of the terms in (\ref{eq:red}) and (\ref{eq:k22n1}). For notational convenience, we write $Z_i=d_n(y_{i+1},y_i)$ and $z_i=\min\{d_n(y_{i+1},y_i),d_n(y_i,y_{i+1})\}=\min\{Z_i,n-Z_i\}$ for $i = 1, 3$ 
	(the minimum distance between vertices $y_{i+1}$ and $y_i$ on a circle with $n$ vertices).
	
	\begin{lemma}\label{lem:mixed} For a simple tripartite-circle drawing of $K_{2,2,n}$ with $n\geq3$, let $Z_i=d_n(y_{i+1},y_i)$ for $i \in \{1,3\}$. The following inequality holds:
		$$2d_n(y_i,x_i)+2d_n(x_{i+1},y_{i+1})+f_n(x_i,x_{i+1})\geq \bigg\lfloor\frac{n}{2}\bigg\rfloor\left\lfloor\frac{n-1}{2}\right\rfloor+n-1-2Z_i.$$
		Moreover, the inequality can be strengthened in the following cases:
		\begin{align*}
			&2d_n(y_i,x_i)+2d_n(x_{i+1},y_{i+1})+f_n(x_i,x_{i+1})\\ &\geq 
			\begin{cases}
				\big\lfloor\frac{n}{2}\big\rfloor\left\lfloor\frac{n-1}{2}\right\rfloor+n-1-2Z_i+2n;
				&\text{if ccw order } \neq x_{i+1}y_{i+1}y_ix_i \\
				\big\lfloor\frac{n}{2}\big\rfloor\left\lfloor\frac{n-1}{2}\right\rfloor+n-1-2Z_i+\left\lfloor\left(Z_i-\frac{n-2}{2}\right)^2\right\rfloor &  \text{if ccw order } = x_{i+1}y_{i+1}y_ix_i \\ &\text{and $Z_i\geq n/2$}
			\end{cases}
		\end{align*}
	\end{lemma}
	
	\begin{proof}
		Note that the (general) statement is equivalent to showing $$2\big(d_n(y_i,x_i)+d_n(x_{i+1},y_{i+1})+d_n(y_{i+1},y_i)\big)+f_n(x_i,x_{i+1})\geq \bigg\lfloor\frac{n}{2}\bigg\rfloor\left\lfloor\frac{n-1}{2}\right\rfloor+n-1.$$
		A short case analysis (by considering the six counterclockwise orders) verifies that for every four points $a,b,c,d$, the following holds:
		\begin{align}\label{eq:ThreeTerms}
			d_n(a,b)+d_n(b,c)+d_n(c,d)
			=\begin{cases}
				d_n(a,d)     &
				\text{if the ccw order is } abcd  \\
				d_n(a,d)+2n    &
				\text{if the ccw order is } adcb\\ 
				d_n(a,d)+n   &\text{otherwise}.
			\end{cases}
		\end{align}
		
		With $a=x_{i+1},b=y_{i+1},c=y_i,d=x_i$, \Cref{eq:ThreeTerms} implies that 
		\begin{align}\label{eq:lowerBound}
			2\big(d_n(x_{i+1},y_{i+1})+d_n(y_{i+1},y_i)+d_n(y_i,x_i)\big)&+f_n(x_i,x_{i+1}) \\ 
			&\geq 2d_n(x_{i+1},x_i) +f_n(x_i,x_{i+1}).\nonumber
		\end{align}
		The right side of Inequality (\ref{eq:lowerBound}) can be expressed as
		the quadratic function
		\begin{equation}\label{eq:min}
			2d_n(x_{i+1}, x_i)+f_n(x_i,x_{i+1}) =  d_n(x_{i+1}, x_i)^2 + (2-n)d_n(x_{i+1}, x_i) + \binom{n}{2},
		\end{equation} 
		which is minimized for $d_n(x_{i+1}, x_i)=\left\lfloor\frac{n-2}{2}\right\rfloor$.
		Evaluation at $d_n(x_{i+1}, x_i)=\left\lfloor\frac{n-2}{2}\right\rfloor$  yields
		\begin{align*}
			2d_n(x_{i+1}, x_i)+f_n(x_i,x_{i+1})
			&\geq\bigg\lfloor\frac{n}{2}\bigg\rfloor\left\lfloor\frac{n-1}{2}\right\rfloor+n -1.
		\end{align*}
		This finishes the proof of the general statement.

		For the strengthening, we consider the two cases.
		If the counterclockwise order is different from $x_{i+1}y_{i+1}y_ix_i$,  note that  we can add at least a $2n$ term to the right side of the Inequality (\ref{eq:lowerBound}).
		If the counterclockwise order is $x_{i+1}y_{i+1}y_ix_i$ and $Z_i\geq n/2$, then $d_n(x_{i+1}, x_i) \ge d_n(y_{i+1},y_i)=Z_i\geq n/2$ and the expression in \Cref{eq:min} is minimized for $d_n(x_{i+1}, x_i)=Z_i$.
	\end{proof}

	Now, we show a lower bound on the remaining four $f_n$-terms in (\ref{eq:k22n1}).
	To do so, we define 
	\[\Delta_k:=k\bmod 2.\]
	Recall that $z_1=\min\{d_n(y_2,y_1),d_n(y_1,y_2)\}$ and $z_3=\min\{d_n(y_4,y_3),d_n(y_3,y_4)\}$.
	\begin{lemma}\label[lemma]{lem:ysNEW} For a simple tripartite-circle drawing of $K_{2,2,n}$ with $n\geq3$, let $S=f_n(y_1,y_3)+f_n(y_1,y_4)+f_n(y_2,y_3)+f_n(y_2,y_4)-4\lfloor\frac{n}{2}\rfloor\left\lfloor\frac{n-1}{2}\right\rfloor$.
		\begin{enumerate}
			\item \label{lem:7.i}If $y_1,y_2\in [y_3,y_4]$ or $y_1,y_2\in [y_4,y_3]$, then it holds that
			\begin{align*}
				S&\geq z_1^2+z_3^2 -\Delta_n\Delta_{z_1+z_3}
			\end{align*}
			\item \label{lem:7.ii}If $y_1\in (y_3,y_4)$ and $y_2\in (y_4,y_3)$ (or vice versa), then it holds that 
			\begin{align*}
				S&\geq \nicefrac{1}{4}\left(z_1^2+(n-z_1)^2+z_3^2+(n-z_3)^2\right)
				-\nicefrac{1}{2}\Delta_n.
			\end{align*}
		\end{enumerate}
		Moreover, in all cases it holds that 
		\begin{align*}
			S\geq 
			z_1^2
			-\Delta_n\Delta_{z_1}.
		\end{align*}
	\end{lemma}
	
	\begin{proof}
		First note that exchanging $y_1$ and $y_2$ (or $y_3$ and $y_4$) does not influence $f_n(y_1,y_3)+f_n(y_1,y_4)+f_n(y_2,y_3)+f_n(y_2,y_4)$.
		Consequently, by swapping $y_1$ and $y_2$ or $y_3$ and $y_4$, all counterclockwise orders can be transformed to one of the following two:  $y_1y_2y_3y_4$ (non-alternating, i.e., \Cref{case:7.i}) and $y_1y_3y_2y_4$ (alternating, i.e., \Cref{case:7.ii})).
		
		\begin{case}\label{case:7.i}Without loss of generality, we consider the counterclockwise order $y_1,y_2,y_3,y_4$ and define $a:=d_n(y_1,y_2)$, $b:=d_n(y_2,y_3)$, $c:=d_n(y_3,y_4)$, and $d:=d_n(y_4,y_1)$. See  \Cref{fig:K22nBOUNDS}a.
			
			\begin{figure}[htb]
				\centering
				\begin{subfigure}[c]{.32\textwidth}
					\centering
					\includegraphics[page=17]{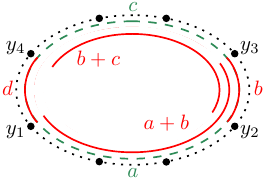}
					\caption{Case \ref{case:7.i}, non-alternating.}
					\label{fig:K22nBOUNDS_A}
				\end{subfigure}\hfil
				\begin{subfigure}[c]{.32\textwidth}
					\centering
					\includegraphics[page=18]{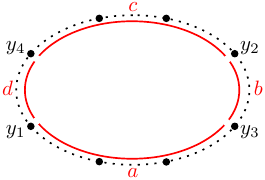}
					\caption{Case \ref{case:7.ii}, alternating.}
					\label{fig:K22nBOUNDS_B}
				\end{subfigure}
				\hfil
				\caption{Illustration of the proof of \Cref{{lem:ysNEW}}.}
				\label{fig:K22nBOUNDS}
			\end{figure}
			
			Then, it holds that
			\begin{align*}
				f_n(y_1,y_3)+&f_n(y_1,y_4)+f_n(y_2,y_3)+f_n(y_2,y_4)\\
				&=4\binom{n}{2}-(a+b)(n-(a+b))-d(n-d)-b(n-b)-(b+c)(n-(b+c))\\
				&= 2n^2-2n + a^2+b^2+c^2+d^2 - (a+b+c+d)n + 2b(a+b+c-n)\\
				&= 2n^2-2n + a^2+b^2+c^2+d^2 - n^2 - 2bd\\
				&=4\bigg\lfloor\frac{n}{2}\bigg\rfloor\left\lfloor\frac{n-1}{2}\right\rfloor-\Delta_n + a^2+c^2 +(b-d)^2
			\end{align*}
			The inequality 
			$ a^2+c^2 +(b-d)^2 \geq z_1^2+z_3^2+\Delta_n\Delta_{z_1+z_3+1}$  holds when $a>z_1$, $c>z_3$, $b\neq d$, $z_1+z_3$ is odd, or $n$ is even. At least one of these conditions holds as $a=z_1,c=z_3,b=d$, and $z_1+z_3$ even imply that $n=z_1+z_3+2b$ is also even. Hence,
			\begin{align*}
				S\geq z_1^2+z_3^2+\Delta_n\Delta_{z_1+z_3+1}-\Delta_n =z_1^2+z_3^2-\Delta_n\Delta_{z_1+z_3}
			\end{align*}
			This finishes the proof of \Cref{case:7.i}.\end{case}
		
		Note that $S\geq z_1^2+z_3^2-\Delta_n\Delta_{z_1+z_3}$ implies that  $S\geq z_1^2-\Delta_n\Delta_{z_1}$ because $z_3^2-\Delta_n\Delta_{z_1+z_3}\geq 0\geq -\Delta_n\Delta_{z_1}$ when $z_3$ is odd, and $\Delta_{z_1+z_3}=\Delta_{z_1}$ when $z_3$ is even. 
		
		\begin{case}\label{case:7.ii}Without loss of generality, we consider the order $y_1,y_3,y_2,y_4$ and define
			$a:=d_n(y_1,y_3)$, $b:=d_n(y_3,y_2)$, $c:=d_n(y_2,y_4)$, and $d:=d_n(y_4,y_1)$. See  \Cref{fig:K22nBOUNDS}b. Then
			\begin{align*}
				f_n(y_1,y_3)+&f_n(y_1,y_4)+f_n(y_2,y_3)+f_n(y_2,y_4)\\&=
				4\binom{n}{2}-a(n-a)-d(n-d)-b(n-b)-c(n-c)
				\\&
				=n^2-2n+a^2+b^2+c^2+d^2
				\\&=
				4\bigg\lfloor\frac{n}{2}\bigg\rfloor\left\lfloor\frac{n-1}{2}\right\rfloor
				-\Delta_n+a^2+b^2+c^2+d^2
			\end{align*}
			Note that 
			$
			a^2+b^2\geq a^2+b^2-\nicefrac{1}{2}\left((a-b)^2-\Delta_{a+b}\right)=\nicefrac{1}{2}\left((a+b)^2+\Delta_{a+b}\right).
			$
			Similarly,
			$c^2+d^2\geq\nicefrac{1}{2}\left((c+d)^2+\Delta_{c+d}\right)$,
			$b^2+c^2\geq\nicefrac{1}{2}\left((b+c)^2+\Delta_{b+c}\right)$, and
			$d^2+a^2\geq\nicefrac{1}{2}\left((d+a)^2+\Delta_{d+a}\right)$.
			We thus have
			\begin{align*}
				S&\geq a^2+b^2+c^2+d^2-\Delta_n\\
				&\geq\nicefrac{1}{4}\left((a+b)^2+(c+d)^2+(b+c)^2+(d+a)^2+\Delta_{a+b}+\Delta_{c+d}+\Delta_{b+c}+\Delta_{d+a}-4\Delta_n\right)\\
				&=\nicefrac{1}{4}\left(z_1^2+(n-z_1)^2+z_3^2+(n-z_3)^2\right)+\nicefrac{1}{4}\left(\Delta_{z_1}+\Delta_{n-z_1}+\Delta_{z_3}+\Delta_{n-z_3}-4\Delta_n\right)\\
				&\geq\nicefrac{1}{4}\left(z_1^2+(n-z_1)^2+z_3^2+(n-z_3)^2\right)	-\nicefrac{1}{2}\Delta_n.
			\end{align*}
			The last inequality is clear if $n$ is even. When $n$ is odd, then for $i=1,3$ one of $z_i$ and $n-z_i$ is even and the other odd. Therefore, it
			holds that 
			$\nicefrac{1}{4}(\Delta_{z_1}+\Delta_{n-z_1}+\Delta_{z_3}+\Delta_{n-z_3}-4\Delta_n)=-\nicefrac{1}{2}\Delta_n.$ 
			Hence, \Cref{case:7.ii} is proved.\end{case}
		
		Finally, note that in \Cref{case:7.ii} it holds that
		$a^2+b^2+c^2+d^2\geq\frac{1}{2}(z_1^2+(n-z_1)^2)\geq z_1^2$. Therefore, $S \geq z_1^2-\Delta_n\Delta_{z_1}$,

		Finally, note that in \Cref{case:7.ii} it holds that
		$$S\geq a^2+b^2+c^2+d^2-\Delta_n\geq\nicefrac{1}{2}(z_1^2+(n-z_1)^2+\Delta_{z_1}+\Delta_{n-z_1})-\Delta_n\geq z_1^2-\Delta_n\Delta_{z_1}.$$
		The last inequality follows from the fact that $n-z_1\geq z_1$ and either $n$ is even or $(n-z_1)^2- z_1^2=n(n-2z_1)\geq 2$. This finishes the proof.
	\end{proof}
	
	Now, we are ready to prove the lower bound for each type of drawing. We begin with the types that are simpler to analyze.
	
	\begin{lemma}\label{lem:type234}
		If $D$ is a simple tripartite-circle drawing of $K_{2,2,n}$ of type $2$, $3$, or $4$ with $n\geq3$, then the number of crossings in $D$ is at least $6\lfloor\frac{n}{2}\rfloor\left\lfloor\frac{n-1}{2}\right\rfloor+2n-3$.
	\end{lemma}
	
	\begin{proof}
		For a drawing of type 2 or 3, by \Cref{eq:red,eq:k22n1} from \Cref{lem:redgreencrossings}, the number of crossings is at least
		\begin{equation*}
			1+2n+f_n(x_1,x_2)+f_n(x_3,x_4)+ f_n(y_1,y_3)+f_n(y_1,y_4)+f_n(y_2,y_3)+f_n(y_2,y_4).
		\end{equation*}
		Using \Cref{eq:k22n3}, this expression is bounded below by $6\lfloor\frac{n}{2}\rfloor\left\lfloor\frac{n-1}{2}\right\rfloor+2n+1>6\lfloor\frac{n}{2}\rfloor\left\lfloor\frac{n-1}{2}\right\rfloor+2n-3$, as desired. Note that drawings of these types do not attain the minimum number of crossings.\medskip
		
		\medskip
		Next, we consider drawings of type 4. By \Cref{lem:redgreencrossings} the number of crossings is
		\begin{align*}
			n&+2(d_n(y_1,x_1)+d_n(x_2,y_2))+f_n(x_1,x_2)
			+f_n(x_3,x_4)\\ &+ f_n(y_1,y_3)+f_n(y_1,y_4)+f_n(y_2,y_3)+f_n(y_2,y_4).
		\end{align*}
		We show the lower bound by considering two cases for $Z_1=d_n(y_{2},y_1)$. In each of the cases, we use \Cref{eq:k22n3} to bound $f_n(x_3,x_4)$.
		
		In the first case, it holds that $Z_1\leq (n-1)/2$. Then $z_1=Z_1$ and by \Cref{lem:mixed,lem:ysNEW}, the number of crossings is at least
		\begin{equation*}
			n+6\bigg\lfloor\frac{n}{2}\bigg\rfloor\left\lfloor\frac{n-1}{2}\right\rfloor+n-1-2Z_1+Z_1^2-\Delta_n\Delta_{Z_1}
		\end{equation*}
		and $2n-1-2Z_1+Z_1^2-\Delta_n\Delta_{Z_1}= 2n-2+(Z_1-1)^2-\Delta_n\Delta_{Z_1} \geq 2n-3$. This shows the claim.\medskip
		
		In the second case, it holds that $Z_1\geq n/2$. Then $z_1=n-Z_1$. Here we distinguish the subcases whether or not $[x_1,x_2]\subseteq[y_1,y_2]$. 
		If $[x_1,x_2]\subseteq[y_1,y_2]$. Using the third inequality of \Cref{lem:mixed}, the number of crossings is at least $6\lfloor\frac{n}{2}\rfloor\left\lfloor\frac{n-1}{2}\right\rfloor+$
		\begin{align*}
			&n+n-1-2Z_1+\left\lfloor\left(Z_1-\frac{n-2}{2}\right)^2\right\rfloor+(n-Z_1)^2-\Delta_n\Delta_{n-Z_1}\\
			&\geq 2n-3+\left\lfloor\nicefrac{1}{8}\big((4Z_1-3n)^2+(n-4)^2\big)\right\rfloor\\
			&\geq 2n-3.
		\end{align*}
		
		If $Z_1\geq n/2$ and $[x_1,x_2]\not\subseteq[y_1,y_2]$, then the second inequality of \Cref{lem:mixed} shows that the number of crossings is at least $6\lfloor\frac{n}{2}\rfloor\left\lfloor\frac{n-1}{2}\right\rfloor+$
		\begin{align*}
			n+n-1-2Z_1+2n+(n-Z_1)^2-\Delta_n\Delta_{n-Z_1}&\\
			\geq & 2n-1+2(n-Z_1)+(n-Z_1)^2-1\\
			=& 2n-3+(n-Z_1+1)^2\\
			\geq & 2n-3.
		\end{align*}
		This finishes the proof for drawings of type 4.
	\end{proof}
	
	\begin{lemma}\label{lem:type1}
		If $D$ is a simple tripartite-circle drawing of $K_{2,2,n}$ of type $1$ with $n\geq3$, then the number of crossings in $D$ is at least $6\lfloor\frac{n}{2}\rfloor\left\lfloor\frac{n-1}{2}\right\rfloor+2n-3$.
	\end{lemma}
	
	\begin{proof}By \Cref{lem:redgreencrossings}, drawings of type 1 have at least the following number of crossings
		\begin{align*}
			&2(d_n(y_1,x_1)+d_n(x_2,y_2)+d_n(y_3,x_3)+d_n(x_4,y_4))+f_n(x_1,x_2)+f_n(x_3,x_4)\\
			&+ f_n(y_1,y_3)+f_n(y_1,y_4)+f_n(y_2,y_3)+f_n(y_2,y_4)+1.
		\end{align*}
		For $i=1,3$,  recall $Z_i=d_n(y_{i+1},y_i)$, $z_i=\min\{Z_i,n-Z_i\}$, and let $$B_i=\begin{cases}
			0& \text{if } Z_i\leq (n-1)/2,\\
			2n& \text{if } Z_i\geq n/2\text{ and ccw order } \neq x_{i+1}y_{i+1}y_ix_i,\\
			\left\lfloor\left(Z_i-\frac{n-2}{2}\right)^2\right\rfloor& \text{if } Z_i\geq n/2\text{ and ccw order } = x_{i+1}y_{i+1}y_ix_i.
		\end{cases}$$ 
		\smallskip
		
		\begin{case} First, we consider the case where \Cref{lem:ysNEW}i) applies. 
			By \Cref{lem:mixed}, \Cref{lem:ysNEW}i), and \Cref{eq:k22n3}, the number of crossings is at least
			\begin{equation*}
				6\bigg\lfloor\frac{n}{2}\bigg\rfloor\left\lfloor\frac{n-1}{2}\right\rfloor+2n-3+(1-2Z_1+z_1^2+B_1)+(1-2Z_3+z_3^2+B_3)-\Delta_n\Delta_{z_1+z_3}.
			\end{equation*}
			We need to prove that $(1-2Z_1+z_1^2+B_1)+(1-2Z_3+z_3^2+B_3)-\Delta_n\Delta_{z_1+z_3}\geq 0$, which follows from the following three inequalities. 
			\begin{itemize}
				\item If $Z_i\leq (n-1)/2$, then $z_i=Z_i$ and $B_i=0$. Thus $$1-2Z_i+z_i^2+B_i=1-2Z_i+Z_i^2=(Z_i-1)^2\geq1$$
				unless $z_i=Z_i=1$.
				\item If $Z_i\geq n/2$ and ccw order  $\neq x_{i+1}y_{i+1}y_ix_i$, then $z_i=n-Z_i$ and $B_i=2n$. Thus
				$$1-2Z_i+z_i^2+B_i=1-2Z_i+(n-Z_i)^2+2n=(n-Z_i+1)^2\geq4$$
				because  $Z_i\leq n-1$.
				\item If $Z_i\geq n/2$ and ccw order  $= x_{i+1}y_{i+1}y_ix_i$, then $z_i=n-Z_i$ and $B_i=	\left\lfloor\left(Z_i-\frac{n-2}{2}\right)^2\right\rfloor$. Thus
				$$1-2Z_i+z_i^2+B_i=1-2Z_i+(n-Z_i)^2+	\left\lfloor\left(Z_i-\frac{n-2}{2}\right)^2\right\rfloor=\left\lfloor\frac{(4Z_3-3n)^2+(n-4)^2}{8}\right\rfloor\geq 1$$
				unless $n=3,4$, or $5$ and $z_i=n-Z_i=1$.
			\end{itemize}
			Note that $1-2Z_1+z_1^2+B_1=1-2Z_3+z_3^2+B_3=0$ if and only if $z_1=z_3=1$ and thus $(1-2Z_1+z_1^2+B_1)+(1-2Z_3+z_3^2+B_3)-\Delta_n\Delta_{z_1+z_3}=0$. In all other cases, $1-2Z_1+z_1^2+B_1$ and $1-2Z_3+z_3^2+B_3$ are both $\geq 0$ and at least one of them is $\geq 1$ so that  $(1-2Z_1+z_1^2+B_1)+(1-2Z_3+z_3^2+B_3)-\Delta_n\Delta_{z_1+z_3}\geq 0$.
		\end{case}
		\smallskip
		
		\begin{case}\label{case:B}
			Now, we consider the case where \Cref{lem:ysNEW}ii) applies; note that this implies $n\geq 4$. Also, $z_i^2+(n-z_i)^2=Z_i^2+(n-Z_i)^2$ for $i=1,3$. By \Cref{lem:mixed}, \Cref{lem:ysNEW}ii), and \Cref{eq:k22n3}, the number of crossings is at least
			\begin{align*}
				6\bigg\lfloor\frac{n}{2}\bigg\rfloor\left\lfloor\frac{n-1}{2}\right\rfloor&+2n-3+\left(1-2Z_1+B_1+\nicefrac{1}{4}(Z_1^2+(n-Z_1)^2-\Delta_n)\right)\\
				&+\left(1-2Z_3+B_3+\nicefrac{1}{4}(Z_3^2+(n-Z_3)^2-\Delta_n)\right).
			\end{align*}
			It is enough to prove that $1-2Z_i+B_i+\nicefrac{1}{4}(Z_i^2+(n-Z_i)^2-\Delta_n)\geq0
			$ for $i=1,3$ in each of the following three cases:
			\begin{itemize}
				\item If $Z_i\leq (n-1)/2$. Then $B_i=0$ and 
				\begin{align*}
					1-2Z_i+B_i&+\nicefrac{1}{4}(Z_i^2 +(n-Z_i)^2-\Delta_n)\\
					&=\nicefrac{1}{4}((Z_i-2)^2+(n-Z_i+2)^2-\Delta_n)-1-n\\
					&\geq \frac{1}{4}\left(\left(\frac{n-5}{2}\right)^2
					+\left(\frac{n+5}{2}\right)^2-\Delta_n\right)-1-n\\
					&=\frac{n^2+25-2\Delta_n}{8}-1-n
					=\frac{(n-4)^2+1-2\Delta_n}{8}\geq0.
				\end{align*}
				\item  If $Z_i\geq n/2$ and $B_i=2n$. Then
				\begin{align*}
					1-2Z_i+B_i&+\nicefrac{1}{4}(Z_i^2+(n-Z_i)^2-\Delta_n)\\
					&	=\nicefrac{1}{4}((Z_i-2)^2+(n-Z_i+2)^2-\Delta_n)-1+n\\
					&\geq \frac{1}{4}\left(\left\lfloor\frac{n}{2}\right\rfloor^2
					+\left\lceil\frac{n}{2}\right\rceil^2-\Delta_n\right)-1+n=\frac{n^2-\Delta_n}{8}-1+n\\
					&\geq\frac{n^2-\Delta_n}{8}+2-n=\frac{(n-4)^2-\Delta_n}{8}\geq0.
				\end{align*}
				\item  If $Z_i\geq n/2$ and $B_i=	\left\lfloor\left(Z_i-\frac{n-2}{2}\right)^2\right\rfloor$. Then
				\begin{align*}
					1-&2Z_i+B_i+\nicefrac{1}{4}(Z_i^2+(n-Z_i)^2-\Delta_n)\\
					&=1-2Z_i+\left\lfloor\left(Z_i-\frac{n-2}{2}\right)^2\right\rfloor+\nicefrac{1}{4}(Z_i^2+(n-Z_i)^2-\Delta_n)\\
					&\geq \frac{1}{4}\left(\left\lfloor\frac{n}{2}\right\rfloor^2
					+\left\lceil\frac{n}{2}\right\rceil^2-\Delta_n\right)+2-n=\frac{n^2-\Delta_n}{8}+2-n=\frac{(n-4)^2-\Delta_n}{8}\geq0.
				\end{align*}
			\end{itemize}
		\end{case}
		This finishes the proof of the lemma.
	\end{proof}
	
	Together, \Cref{le:k22n,lem:type1,lem:type234} yield the lower bound.
	
	\LowerBound*
	
	\begin{proof}
		Let $D$ be a tripartite-circle drawing of $K_{2,2,n}$ with the minimum number of crossings among all such drawings. By crossing-minimality, $D$ is simple. By \Cref{le:k22n}, either $D$ is of type $1$, and \Cref{lem:type1} applies, or $D$ is of type $2$, $3$, or $4$, and \Cref{lem:type234} applies. In all cases, $D$ has at least $6\lfloor\frac{n}{2}\rfloor\left\lfloor\frac{n-1}{2}\right\rfloor+2n-3$ crossings.
	\end{proof}
	
	
	\section{Conclusion}\label{sec:open}
	In this paper, we determined the exact value of the tripartite-circle crossing number of  $K_{m,n,p}$ for all $m,n\leq 2$. The natural next questions in this direction include determining $\crN{3}(K_{m,n,p})$ for other small values of $m$ and $n$.
	Of particular interest are the values $(m,n) = (1,3)$, $(1,4)$, $(1,5)$, $(2,3)$ because in these cases the exact value of $\crg(K_{m,n,p})$ is known but the exact value of $\crN{3}(K_{m,n,p})$ is not. 
	
	Another natural extension of this work is to study $k$-partite-circle drawings for $k > 3$. In particular, what is the  $k$-partite-circle crossing number of $K_{2,\ldots,2,n}$ for $k > 3$?
	
	One implication of our result is that $\crg(K_{2,4,n})- \crN{3}(K_{2,2,n}) = 3$. It is an open problem whether there is a connection between drawings of $K_{2,4,n}$ and tripartite-circle drawings of $K_{2,2,n}$ that implies this similarity in the different crossing numbers.

	\subsection*{Acknowledgements}
	
	This material is based upon work supported by the National Science Foundation under Grant Number DMS 1641020.
	Silvia Fern\'andez-Merchant was supported by the NSF grant DMS 1400653.  Marija Jeli\'c Milutinovi\'c has been supported by the Project No.\ 7744592 MEGIC ``Integrability and Extremal Problems in Mechanics, Geometry and Combinatorics'' of the Science Fund of Serbia, and by the Faculty of Mathematics University of Belgrade through the grant (No.\ 451-03-47/2023-01/200104) by the Ministry of Education, Science, and Technological Development of the Republic of Serbia. Rachel Kirsch was partially supported by NSF grant DMS 1839918.
	
	\providecommand{\arxiv}[1]{\href{http://www.arxiv.org/abs/#1}{arXiv~#1}}
	\providecommand{\doi}[1]{\url{https://doi.org/#1}}
	\providecommand{\href}[2]{#2}

\end{document}